\documentclass[12pt,a4paper,reqno]{amsart}
\usepackage[utf8]{inputenc}
\usepackage{enumerate}
\usepackage{mathtools}
\usepackage{mathrsfs}
\usepackage{soul, xcolor}
\usepackage{wasysym}
\usepackage{amsmath,amssymb,xpatch,relsize,graphics,graphicx,amsthm}
\usepackage{float,graphicx}
\usepackage{subcaption}

\xapptocmd\normalsize{%
	\abovedisplayskip=12pt plus 3pt minus 9pt
	\abovedisplayshortskip=0pt plus 3pt
	\belowdisplayskip=12pt plus 3pt minus 9pt
	\belowdisplayshortskip=7pt plus 3pt minus 4pt
}{}{}
\theoremstyle{definition}
\newtheorem{definition}{Definition}[section]
\newtheorem{remark}[definition]{Remark}

\theoremstyle{plain}
\newtheorem{theorem}[definition]{Theorem}
\newtheorem{corollary}[definition]{Corollary}
\newtheorem{lemma}[definition]{Lemma}

\numberwithin{equation}{section}

\title[Starlikeness using special functions and subordination]{\textbf{Starlikeness using special functions and subordination}}
\author[M. Sharma]{Meghna Sharma}
\address{Department of Mathematics, University of Delhi, Delhi--110 007, India}
\email{meghnasharma203@gmail.com}
\author[N.K. Jain]{Naveen Kumar Jain}
\address {Department of Mathematics, Aryabhatta College, Delhi-110021,India}
\email{naveenjain05@gmail.com}
\author[S. Kumar]{Sushil Kumar}
\address {Bharati Vidyapeeth's college of Engineering, Delhi-110063, India}
\email{sushilkumar16n@gmail.com}

\date{}



\thanks{The first author is supported by Senior Research Fellowship from Council of Scientific and Industrial Research, New Delhi, Ref. No.:1753/(CSIR-UGC NET JUNE, 2018).}

\allowdisplaybreaks

\begin{document}
	\maketitle
	\begin{abstract}
		The association of subordination and special functions is used to find sharp estimates on the parameter $\beta$ such that the analytic function $p(z)$ is subordinate to certain functions having positive real part whenever $p(z)+\beta z p'(z)$ is subordinate to the Janowski function. Further, when the traditional approach of solving higher order differential subordination implications failed, the concept of admissibility is employed to establish certain second and third order differential subordination relations between the analytic function $p$ and the functions associated with right half plane. As a sequel, we demonstrated the starlikeness of various well-known analytic functions as well.
	\end{abstract}
	
	\section{Introduction}
	Let $\mathcal{A}$ be the class of analytic functions $f$ defined on the open unit disc $\mathbb{D}:=\{z \in \mathbb{C}: |z|<1\}$ normalized by the conditions $f(0)=f'(0)-1=0$ and $\mathscr{U}$ denote the collection of $univalent$ functions.
	The analytic function $f_1$ is $subordinate$ to another analytic function $f_2$ if there is a Schwarz function $w$ such that $f_1(z)=(f_2ow)(z)$.
	The theory of subordination developed by Miller and Mocanu \cite{MR0506307} has been used extensively in proving several known results involving analytic functions.
	We denote the class of $starlike$ \cite{MR1343506} functions by $\mathcal{S}^{*}$ which comprises of all those analytic functions $f$ satisfying the relation $\operatorname{Re}(zf'(z)/f(z))>0$ for $z \in \mathbb{D}$.
	Later, Ma-Minda unified various subclasses of $starlike$ and presented a more generalized version associated with the function $\varphi$ as follows:
	\[\mathcal{S}^{*}_{\varphi}:=\left\{f \in \mathcal{A}:\frac{zf'(z)}{f(z)} \prec \varphi(z); z \in \mathbb{D}\right\}.\]
	The function $\varphi \in \mathscr{U}$ is analytic, having positive real part and image of unit disc $\mathbb{D}$ under $\varphi$ is starlike with respect to $\varphi(0)=1$.
	Numerous authors have introduced and studied several prominent subclasses of $starlike$ functions for the appropriate choices of the function $\varphi$.
	For instance,
	\begin{enumerate}[(1)]
	\item $\mathcal{S}^{*}[A,B]:=\mathcal{S}^{*}((1+Az)/(1+Bz))$, for $-1 \leq B < A \leq 1$, is the subclass of $Janowski$ \cite{MR0267103} starlike functions.
	\item $S^{*}_{e}:=S^{*}(e^{z})$ is the subclass associated with $exponential$ \cite{MR3394060} function $e^z$.
	\item $\mathcal{S}^{*}_{R}:=\mathcal{S}^{*}(\phi_{0})$, where
	$\phi_{0}(z):=1+(z/k)((k+z)/(k-z)), k=\sqrt{2}+1$ is the subclass associated with the $rational$ \cite{MR3496681} function $\phi_{0}$.
	\item $\mathcal{S}^{*}_{c}:=\mathcal{S}^{*}(\phi_{c})$, where
	$\phi_{c}(z):=2z^2/3+4z/3+1$ is the class associated with a heart shaped $cardioid$ \cite{MR3536076} which is represented by the equation $9(u^2+v^2)-6u+1=1/16(9(u^2+v^2)-18u+5)^2$.
	\item $\mathcal{S}^{*}_{q}:=\mathcal{S}^{*}(\phi_{q}(z)=z+\sqrt{1+z^{2}})$ is the subclass associated with $lune$ \cite{MR3469339}.
	\item $\mathcal{S}^{*}_{L}:=\mathcal{S}^{*}[\sqrt{1+z}]$ is the subclass associated with $lemniscate$ \cite{lemniscate} of Bernoulli.
	\item $\mathcal{S}^{*}_{s}:=\mathcal{S}^{*}(\phi_{s}(z)=1+\sin z)$ is the subclass associated with the $Sine$ \cite{sine} function.
	\item $\mathcal{S}^{*}_{B}:=\mathcal{S}^{*}(e^{e^{z-1}})$ is the subclass associated with $bell$ \cite{MR4017390} numbers.
	\item $\mathcal{S}^{*}_{Ne}:=\mathcal{S}^{*}(\phi_{Ne}(z)=1+z-z^3/3)$ is the subclass associated with a 2-cusped kidney shaped $nephroid$ \cite{MR4190740} domain.
	\item $\mathcal{S}^{*}_{\tanh}:=\mathcal{S}^{*}(1+\tanh z)$ is the subclass associated with $tangent \: hyperbolic$ \cite{MR4354938} function.
	\item $\mathcal{S}^{*}_{SG}:=\mathcal{S}^{*}(\phi_{SG})$, where $\phi_{SG}(z):=2/(1+e^{-z})$ is the subclass associated with the $modified \enspace sigmoid$ \cite{MR4044913} function.
	\end{enumerate}
For recent  various results related to above mentioned subclasses, see \cite{Jain1, jain2, Jain3}.

The class $\mathcal{P}$ is the class of analytic functions having positive real part.
In $1989$, authors \cite{MR0975653} showed that $p(z) \prec 1+z$ whenever $1+zp'(z) \prec 1+z$.
In \cite{MR2336133}, authors determined some non sharp estimates such that $p(z) \prec (1+Az)/(1+Bz)$, where $A,B,D,E \in [-1,1]$ whenever
$1 + \beta z p'(z)/p^{j}(z) \prec (1+Dz)/(1+Ez)$, $j=0,1,2$. Recently, Ali \emph{et al.\@} \cite{Jain4} studied the class $p(z) \prec (1+Az)/(1+Bz)$ to obtain the various results.
Improving the earlier methods used, Kumar and Ravichandran \cite{MR3800966} later obtained sharp bounds on $\beta$ such that $1 + \beta z p'(z)/p^{j}(z),j=0,1,2$ is subordinate to functions in $\mathcal{P}$ implies
$p(z)$ is subordinate to exponential and Janowski function.
Further, authors \cite{meghna-subordination} determined sharp estimates on $\beta$ so that
$1 + \beta zp'(z)/p^{j}(z) \prec \phi_{SG}(z), \phi_{0}(z) \text{ or } \phi_{c}(z)$ for $j=0,1,2$
implies $p(z)$ is subordinate to some analytic function.
Recently, authors~\cite{MR4447482} obtained conditions on $\beta$ such that $p(z) \prec e^{z}$ whenever $1+\beta z p'(z)$ is subordinate to $\sqrt{1+cz}$ or $1+\sqrt{2}z+z^2/2$.
Special functions appear as the solutions of various differential equations related to our real life problems.
They have been used extensively for the last few years in connection to the univalent function theory. In 2019, authors \cite{MR4025222} determined estimates on $\beta$ such that $p(z)+\beta z p(z)$ or $1+\beta z p'(z)$ is subordinate to certain analytic functions with positive real part using certain special functions. For related papers, refer~\cite{ahuja18,bohra21,MR4420664}.

Let $\Omega$ and $\Delta$ be the improper sets in the complex plane and $\psi:\mathbb{C}^3 \times \mathbb{D} \to \mathbb{D}$ be the function.
The search of finding the existence of the smallest set $\Delta \subset \mathbb{C}$ and the largest set $\Omega \subset \mathbb{C}$ such that the implication
$\psi(p(z), zp'(z), z^{2}p''(z); z) \subset \Omega \implies p(z) \subset \Delta; z \in \mathbb{D}$
holds, gave rise to the concept of admissibility conditions.
Several authors have been implementing this approach in the subordination theory to find sufficient conditions for starlikeness, convexity, boundedness and other geometric properties.

In the first theorem of this paper, we use the theory of hypergeometric functions to obtain sharp estimates on $\beta$  so that whenever $p(z)+\beta z p'(z)$ is subordinate to $Janowski$ function implies $p(z)$ is subordinate to certain functions with positive real part. Subordination relations for the particular cases are demonstrated graphically. As a consequence, sufficient conditions for certain type of starlikeness of the function $f \in \mathcal{A}$ are also derived. Let us recall some preliminaries in order to obtain our required bounds.
The admissibility conditions are the key tool which is used  to solve some second and third order differential subordination relations. This is achieved by extending the pre existing literature on admissibility conditions for the $exponential$ function.
	
	\begin{definition}
		\textbf{(Gauss Hypergeometric Function):}
		For $|z|<1$ and parameters $a,b \in \mathbb{C}$ and $c \notin \{0 \cup \mathbb{Z}_{-}\}$, the hypergeometric function ${}_2F_1(a,b;c;z)$ is defined by the convergent power series
		\begin{equation}\label{hypergeometricfunction}
		F(a,b;c;z)= {}_2F_1(a,b;c;z)=\sum_{k=0}^{\infty}\frac{(a)_{k}(b)_{k}}{(c)_{k}k!}z^{k}.
		\end{equation}
		\end{definition}
	The function $F(a,b;c;z)$ is analytic in $\mathbb{C}$ and is one of the solutions of the differential equation
	\[z(1-z)y''+[c-(a+b+1)z]y'-aby=0\]
	at $z=0$.
	The differential of the function $F(a,b;c;z)$ satisfies the relation
	\[\frac{\partial F(a,b;c;z)}{\partial z}=\frac{ab}{c} F(a+1,b+1;c+1;z).\]
	
	\begin{definition}
	\textbf{(Order of Starlikeness):}
	For an analytic function $f$, its order of starlikeness with respect to zero is defined as follows:
	\[\sigma(f):=\inf_{z \in \mathbb{D}}\operatorname{Re}\left(\frac{zf'(z)}{f(z)}\right) \in [-\infty,1]\]
	\end{definition}

In the proof of our first result, the following results are used:
\begin{lemma}\label{subordinationlemma}\cite[Theorem 3.4h, p.132]{MR1760285}
	Let $q : \mathbb{D} \rightarrow \mathbb{C}$ be analytic and $\psi$ and $v$ be analytic in a domain $U \supseteq q(\mathbb{D})$ with $\psi(w) \neq 0$ whenever $w \in q(\mathbb{D})$. Set
	\[Q(z):= zq'(z)\psi(q(z)) \quad \text{and} \quad h(z):=v(q(z))+Q(z),z \in \mathbb{D}.\]
	Suppose that
	\begin{enumerate}[(i)]
		\item either $h(z)$  is convex, or $Q(z)$ is starlike univalent in $\mathbb{D}$ and
		\item $\operatorname{Re}\left(\frac{zh'(z)}{Q(z)}\right)>0, z \in \mathbb{D}$.
	\end{enumerate}
	If $p$ is analytic in $\mathbb{D}$, with $p(0)=q(0)$, $p(\mathbb{D}) \subset U$ and
	\[v(p(z))+zp'(z)\psi(p(z)) \prec v(q(z))+zq'(z)\psi(q(z))\]
	then $p \prec q$, and $q$ is the best dominant.
\end{lemma}

	\begin{theorem}\cite[Theorem 1(a)]{MR2338668}\label{thmforstralikeness}
		Let $a,b$ and $c$ be non-zero real numbers such that $0<a \leq b \leq c$. Then,
		\[1-\frac{ab}{b+c} \leq \sigma(zF(a,b;c;z)) \leq 1-\frac{ab}{2c}.\]
	\end{theorem}

\section{Main Results}
    First result of this section provides sharp estimates on $\beta$ so that the subordination $p(z)+\beta z p'(z) \prec (1+Az)/(1+Bz)$ implies $p(z)$ is subordinate to various well-defined functions with positive real part like $e^z$, $\sqrt{1+z}$, $1+\tanh(z)$, $e^{e^{z-1}}$, $\phi_{s}(z)$, $\phi_{q}(z)$, $\phi_{c}$, $\phi_{0}(z)$ and $\phi_{Ne}(z)$ which ensure various types of starlikeness.

    \begin{theorem}
    	We assume
    	\begin{equation}\label{chi}
    	\chi(\beta,A,B):=-\frac{A-B}{\beta+1}\sum_{j=0}^{\infty}\frac{\Gamma(j)}{(j-1)!(1+\beta+j\beta)}B^{j}+\frac{\beta}{\beta+1}+\frac{1}{\beta+1}
    	\end{equation}
    	and
    	\begin{equation}\label{xi}
    	\xi(\beta,A,B):=\frac{(A-B)}{\beta+1}\sum_{j=0}^{\infty}\frac{\Gamma(j)}{(j-1)!(1+\beta+j\beta)}(-B)^{j}-\frac{\beta}{\beta+1}-\frac{1}{\beta+1}
    	\end{equation}
    	where $-1 \leq B < A \leq 1$.
    	Let $p$ be an analytic function defined on $\mathbb{D}$ with $p(0)=1$ and satisfying
    	\[p(z)+\beta z p'(z) \prec \frac{1+Az}{1+Bz}.\]
    	If $\beta \geq \max\{\beta_{1},\beta_{2}\}$, then
    	\begin{enumerate}[(a)]
    	\item $p(z) \prec e^{z}$, where $\beta_{1}$ and $\beta_{2}$ are positive roots of equations
    	\begin{equation}\label{e-root}
    		e\chi(\beta,A,B) = 1 \enspace \text{and} \enspace \xi(\beta,A,B) = e \enspace \text{respectively.}	
    	\end{equation}
    	
    	\item $p(z) \prec \phi_{s}(z)$, where $\beta_{1}$ and $\beta_{2}$ are positive roots of equations
    	\begin{equation}\label{sin-root}
    		\chi(\beta,A,B)+\sin 1 = 1 \enspace \text{and} \enspace \xi(\beta,A,B)-\sin 1 = 1 \enspace \text{respectively.}	
    	\end{equation}

    \item $p(z) \prec \phi_{0}(z)$, where $\beta_{1}$ and $\beta_{2}$ are positive roots of equations
    \begin{equation}\label{phi0-root}
    	\chi(\beta,A,B)=2(\sqrt{2}-1) \enspace \text{and} \enspace \xi(\beta,A,B) = 2 \enspace \text{respectively.}	
    \end{equation}

\item $p(z) \prec \phi_{c}(z)$, where $\beta_{1}$ and $\beta_{2}$ are positive roots of equations
\begin{equation}\label{phic-root}
	3 \chi(\beta,A,B) = 1 \enspace \text{and} \enspace \xi(\beta,A,B) = 3 \enspace \text{respectively.}	
\end{equation}

\item $p(z) \prec \phi_{Ne}(z)$, where $\beta_{1}$ and $\beta_{2}$ are positive roots of equations
\begin{equation}\label{phiNe-root}
	3\chi(\beta,A,B) = 1 \enspace \text{and} \enspace 3\xi(\beta,A,B) = 5 \enspace \text{respectively.}	
\end{equation}

\item $p(z) \prec \phi_{q}(z)$, where $\beta_{1}$ and $\beta_{2}$ are positive roots of equations
\begin{equation}\label{phiq-root}
	\chi(\beta,A,B)=\sqrt{2}-1 \enspace \text{and} \enspace \xi(\beta,A,B) = \sqrt{2}+1 \enspace \text{respectively.}	
\end{equation}

\item $p(z) \prec \sqrt{1+z}$, where $\beta_{1}$ and $\beta_{2}$ are positive roots of equations
\begin{equation}\label{phil-root}
	\chi(\beta,A,B) = 0 \enspace \text{and} \enspace \xi(\beta,A,B) = \sqrt{2} \enspace \text{respectively.}	
\end{equation}

\item $p(z) \prec e^{e^{z}-1}$, where $\beta_{1}$ and $\beta_{2}$ are positive roots of equations
\begin{equation}\label{bell-root}
	\chi(\beta,A,B)=e^{e^{-1}-1} \enspace \text{and} \enspace \xi(\beta,A,B) = e^{e-1} \enspace \text{respectively.}	
\end{equation}

\item $p(z) \prec 1+\tanh z$, where $\beta_{1}$ and $\beta_{2}$ are positive roots of equations
\begin{equation}\label{tan-root}
	\chi(\beta,A,B)+\tanh 1= 1 \enspace \text{and} \enspace \xi(\beta,A,B)-\tanh 1 = 1 \enspace \text{respectively.}	
\end{equation}
\end{enumerate}
All the estimates on $\beta$ are sharp.
\end{theorem}

\begin{proof}
	Let the analytic function
	\[q_{\beta}(z):=\frac{A-B}{\beta+1}z \left(F(1,1+\frac{1}{\beta};2+\frac{1}{\beta};-Bz)\right)+\frac{\beta}{\beta+1}+\frac{1}{\beta+1}\]
	be the solution of the differential equation
	\[\frac{dq}{dz}+\frac{1}{\beta z}q=\frac{1}{\beta z}\left(\frac{1+Az}{1+Bz}\right).\]
	For $w \in \mathbb{C}$, we define $v(w):=w$ and $\psi(w):=\beta$.
	Let
	\begin{align*}
		Q(z)&=z q_{\beta}'(z)\psi(q_{\beta}(z))\\
		&=\beta z q_{\beta}'(z)\\
		&=\beta z \biggr[\frac{A-B}{\beta+1} \left(F(1,1+\frac{1}{\beta};2+\frac{1}{\beta};-Bz)\right)+\frac{A-B}{2\beta+1}z\left(F(2,2+\frac{1}{\beta};3+\frac{1}{\beta};-Bz)\right)\biggr].
	\end{align*}
	From the hypergeometric functions $F(a,b;c;z)$ defined in (\ref{hypergeometricfunction}) and $F(2,2+\frac{1}{\beta};3+\frac{1}{\beta};-Bz)$, we have $a=2, b=2+\frac{1}{\beta}$ and $c=3+\frac{1}{\beta}$.
	Hence, the hypothesis $0<a\leq b \leq c$ as given in Theorem (\ref{thmforstralikeness}) holds.
	Since $\beta>0,$ then
	\[\sigma\left(zF(2,2+\frac{1}{\beta};3+\frac{1}{\beta};-Bz)\right) \geq 1-\frac{2+4\beta}{2+5\beta}=\frac{\beta}{2+5\beta}>0.\]
	Therefore, the  function $zF(2,2+\frac{1}{\beta};3+\frac{1}{\beta};-Bz)$ is starlike that ensures the starlikeness of  $Q.$
	Also, the function $h$ is defined as
	$h(z)=v(q_{\beta}(z))+Q(z)=q_{\beta}(z)+Q(z)$
	which satisfies
	\begin{align*}
	\operatorname{Re}\left(\frac{zh'(z)}{Q(z)}\right)
	&=\frac{1}{\beta}+\operatorname{Re}\left(\frac{zQ'(z)}{Q(z)}\right).
	\end{align*}
	Since $\beta >0$ and $Q$ is starlike, $\operatorname{Re}\left(zh'(z)Q(z)\right)>0$ for all $z \in \mathbb{D}$.
	Further, using the Lemma \ref{subordinationlemma}, it is noted that
	$p(z)+\beta z p'(z) \prec q_{\beta}(z)+\beta z q_{\beta}'(z)$
	implies the subordination $p \prec q_{\beta}$.
	For the appropriate choice of function $\mathcal{P}(z)$, the subordination $q_{\beta}(z) \prec \mathcal{P}(z)$ holds if the following inequalities are satisfied:
	\begin{equation}\label{sufficient}
		\mathcal{P}(-1) \leq q_{\beta}(-1) \leq q_{\beta}(1) \leq \mathcal{P}(1).
	\end{equation}
	Since the subordination is transitive, it is enough to show that $q_{\beta}(z) \prec \mathcal{P}(z)$ for the required subordination $p(z) \prec \mathcal{P}(z)$ to hold.
	The condition given in \eqref{sufficient} is necessary as well as sufficient for the subordination $p \prec \mathcal{P}$ to hold.
	It is noted that
	\[q_{\beta}(-1)=-\frac{A-B}{\beta+1} \left(F(1,1+\frac{1}{\beta};2+\frac{1}{\beta};B)\right)+\frac{\beta}{\beta+1}+\frac{1}{\beta+1}\] and
	\[q_{\beta}(1)=\frac{A-B}{\beta+1} \left(F
	(1,1+\frac{1}{\beta};2+\frac{1}{\beta};-B)\right)+\frac{\beta}{\beta+1}+\frac{1}{\beta+1}.\]
	\begin{enumerate}[(a)]
		\item For $\mathcal{P}(z)=e^{z}$, the subordination $q_{\beta} \prec e^{z}$ holds if
		\[e^{-1} \leq q_{\beta}(-1) \leq q_{\beta}(1) \leq e.\]
		Thus, these inequalities gives
		\[e^{-1} \leq -\frac{A-B}{\beta+1} \left(F(1,1+\frac{1}{\beta};2+\frac{1}{\beta};B)\right)+\frac{\beta}{\beta+1}+\frac{1}{\beta+1}\] and
		\[e \geq \frac{A-B}{\beta+1} \left(F(1,1+\frac{1}{\beta};2+\frac{1}{\beta};-B)\right)+\frac{\beta}{\beta+1}+\frac{1}{\beta+1}.\]
		Consequently, these inequalities reduces to
		\[
		-\frac{A-B}{\beta+1}\sum_{j=0}^{\infty}\frac{\Gamma (j)}{(j-1)!(1+\beta+j\beta)}B^{j}+\frac{\beta}{\beta+1}+\frac{1}{\beta+1}-\frac{1}{e} \geq 0
		\] and
		\[
		e-\frac{(A-B)}{\beta+1}\sum_{j=0}^{\infty}\frac{\Gamma (j)}{(j-1)!(1+\beta+j\beta)}(-B)^{j}-\frac{\beta}{\beta+1}-\frac{1}{\beta+1} \geq 0.
		\]
		respectively.
		Therefore, in view of (\ref{chi}) and (\ref{xi}), the desired subordination $q_{\beta} \prec e^{z}$ holds if $\beta \geq \max\{\beta_1, \beta_2\}$, where $\beta_1$ and $\beta_2$ are positive roots of the equations given in (\ref{e-root}).
		
		\item For $\mathcal{P}(z)=\phi_{s}(z)$, the subordination $q_{\beta} \prec \phi_{s}$ holds if
		\[\phi_{s}(-1) \leq q_{\beta}(-1) \leq q_{\beta}(1) \leq \phi_{s}(1).\]
		Accordingly, these inequalities gives
		\[1-\sin 1 \leq -\frac{A-B}{\beta+1} \left(F(1,1+\frac{1}{\beta};2+\frac{1}{\beta};B)\right)+\frac{\beta}{\beta+1}+\frac{1}{\beta+1}\] and
		\[1+\sin 1 \geq \frac{A-B}{\beta+1} \left(F(1,1+\frac{1}{\beta};2+\frac{1}{\beta};-B)\right)+\frac{\beta}{\beta+1}+\frac{1}{\beta+1}.\]
		or equivalently,
		\[-\frac{A-B}{\beta+1}\sum_{j=0}^{\infty}\frac{\Gamma (j)}{(j-1)!(1+\beta+j\beta)}B^{j}+\frac{\beta}{\beta+1}+\frac{1}{\beta+1}-1+\sin 1 \geq 0
		\] and
		\[1+\sin 1-\frac{(A-B)}{\beta+1}\sum_{j=0}^{\infty}\frac{\Gamma (j)}{(j-1)!(1+\beta+j\beta)}(-B)^{j}-\frac{\beta}{\beta+1}-\frac{1}{\beta+1} \geq 0.
		\]
		Therefore, the required subordination $q_{\beta} \prec \phi_{s}$ holds if $\beta \geq \max\{\beta_1, \beta_2\}$, where $\beta_1$ and $\beta_2$ are positive roots of the equations given in (\ref{sin-root}).
		
		\item For $\mathcal{P}(z)=\phi_{0}(z)$, the subordination $q_{\beta} \prec \phi_{0}$ holds if
		\[\phi_{0}(-1) \leq q_{\beta}(-1) \leq q_{\beta}(1) \leq \phi_{0}(1).\]
		The above inequalities reduces to
		\[2(\sqrt{2}-1) \leq -\frac{A-B}{\beta+1} \left(F(1,1+\frac{1}{\beta};2+\frac{1}{\beta};B)\right)+\frac{\beta}{\beta+1}+\frac{1}{\beta+1}\] and
		\[2 \geq \frac{A-B}{\beta+1} \left(F(1,1+\frac{1}{\beta};2+\frac{1}{\beta};-B)\right)+\frac{\beta}{\beta+1}+\frac{1}{\beta+1}.\]
		or equivalently,
		\[
		-\frac{A-B}{\beta+1}\sum_{j=0}^{\infty}\frac{\Gamma (j)}{(j-1)!(1+\beta+j\beta)}B^{j}+\frac{\beta}{\beta+1}+\frac{1}{\beta+1}-2\sqrt{2}+2 \geq 0
		\] and
		\[
			2-\frac{(A-B)}{\beta+1}\sum_{j=0}^{\infty}\frac{\Gamma (j)}{(j-1)!(1+\beta+j\beta)}(-B)^{j}-\frac{\beta}{\beta+1}-\frac{1}{\beta+1} \geq 0.
		\]
		Therefore, the desired subordination $q_{\beta} \prec \phi_{0}$ holds if $\beta \geq \max\{\beta_1, \beta_2\}$, where $\beta_1$ and $\beta_2$ are positive roots of the equations given in (\ref{phi0-root}).
		
		\item For $\mathcal{P}(z)=\phi_{c}(z)$, the subordination $q_{\beta} \prec \phi_{c}$ holds if
		\[\phi_{c}(-1) \leq q_{\beta}(-1) \leq q_{\beta}(1) \leq \phi_{c}(1).\]
		The above inequalities reduces to
		\[\frac{1}{3} \leq -\frac{A-B}{\beta+1} \left(F(1,1+\frac{1}{\beta};2+\frac{1}{\beta};B)\right)+\frac{\beta}{\beta+1}+\frac{1}{\beta+1}\] and
		\[3 \geq \frac{A-B}{\beta+1} \left(F(1,1+\frac{1}{\beta};2+\frac{1}{\beta};-B)\right)+\frac{\beta}{\beta+1}+\frac{1}{\beta+1}.\]
		or equivalently,
		\[-\frac{A-B}{\beta+1}\sum_{j=0}^{\infty}\frac{\Gamma (j)}{(j-1)!(1+\beta+j\beta)}B^{j}+\frac{\beta}{\beta+1}+\frac{1}{\beta+1}-\frac{1}{3} \geq 0\]
		and
		\[3-\frac{(A-B)}{\beta+1}\sum_{j=0}^{\infty}\frac{\Gamma (j)}{(j-1)!(1+\beta+j\beta)}(-B)^{j}-\frac{\beta}{\beta+1}-\frac{1}{\beta+1} \geq 0.\]
		Therefore, the required subordination $q_{\beta} \prec \phi_{c}$ holds if $\beta \geq \max\{\beta_1, \beta_2\}$, where $\beta_1$ and $\beta_2$ are positive roots of the equations given in (\ref{phic-root}).
		
		
		\item For $\mathcal{P}(z)=\phi_{Ne}(z)$, the subordination $q_{\beta} \prec \phi_{Ne}$ holds if
		\[\phi_{Ne}(-1) \leq q_{\beta}(-1) \leq q_{\beta}(1) \leq \phi_{Ne}(1).\]
		The above inequalities reduces to
		\[\phi_{Ne}(-1) \leq -\frac{A-B}{\beta+1} \left(F(1,1+\frac{1}{\beta};2+\frac{1}{\beta};B)\right)+\frac{\beta}{\beta+1}+\frac{1}{\beta+1}\] and
		\[\phi_{Ne}(1) \geq \frac{A-B}{\beta+1} \left(F(1,1+\frac{1}{\beta};2+\frac{1}{\beta};-B)\right)+\frac{\beta}{\beta+1}+\frac{1}{\beta+1}.\]
		or equivalently,
		\[
			-\frac{A-B}{\beta+1}\sum_{j=0}^{\infty}\frac{\Gamma (j)}{(j-1)!(1+\beta+j\beta)}B^{j}+\frac{\beta}{\beta+1}+\frac{1}{\beta+1}-\frac{1}{3} \geq 0
		\] and
		\[
			\frac{5}{3}-\frac{(A-B)}{\beta+1}\sum_{j=0}^{\infty}\frac{\Gamma (j)}{(j-1)!(1+\beta+j\beta)}(-B)^{j}-\frac{\beta}{\beta+1}-\frac{1}{\beta+1} \geq 0.
		\]
		Therefore, the required subordination $q_{\beta} \prec \phi_{Ne}$ holds if $\beta \geq \max\{\beta_1, \beta_2\}$, where $\beta_1$ and $\beta_2$ are positive roots of the equations given in (\ref{phiNe-root}).
		
		\item For $\mathcal{P}(z)=\phi_{q}(z)$, the subordination $q_{\beta} \prec \phi_{q}$ holds if
		\[\phi_{q}(-1) \leq q_{\beta}(-1) \leq q_{\beta}(1) \leq \phi_{q}(1).\]
		The above inequalities reduces to
		\[\phi_{q}(-1) \leq -\frac{A-B}{\beta+1} \left(F(1,1+\frac{1}{\beta};2+\frac{1}{\beta};B)\right)+\frac{\beta}{\beta+1}+\frac{1}{\beta+1}\] and
		\[\phi_{q}(1) \geq \frac{A-B}{\beta+1} \left(F(1,1+\frac{1}{\beta};2+\frac{1}{\beta};-B)\right)+\frac{\beta}{\beta+1}+\frac{1}{\beta+1}.\]
		or equivalently,
		\[
			-\frac{A-B}{\beta+1}\sum_{j=0}^{\infty}\frac{\Gamma (j)}{(j-1)!(1+\beta+j\beta)}B^{j}+\frac{\beta}{\beta+1}+\frac{1}{\beta+1}-(-1+\sqrt{2}) \geq 0
		\] and
		\[
			1+\sqrt{2}-\frac{(A-B)}{\beta+1}\sum_{j=0}^{\infty}\frac{\Gamma (j)}{(j-1)!(1+\beta+j\beta)}(-B)^{j}-\frac{\beta}{\beta+1}-\frac{1}{\beta+1} \geq 0.
		\]
		Therefore, the desired subordination $q_{\beta} \prec \phi_{q}$ holds if $\beta \geq \max\{\beta_1, \beta_2\}$, where $\beta_1$ and $\beta_2$ are positive roots of the equations given in (\ref{phiq-root}).
		
		\item For $\mathcal{P}(z)=\sqrt{1+z}$, the subordination $q_{\beta} \prec \sqrt{1+z}$ holds if
		\[0 \leq q_{\beta}(-1) \leq q_{\beta}(1) \leq \sqrt{2}.\]
		The above inequalities reduces to
		\[0 \leq -\frac{A-B}{\beta+1} \left(F(1,1+\frac{1}{\beta};2+\frac{1}{\beta};B)\right)+\frac{\beta}{\beta+1}+\frac{1}{\beta+1}\] and
		\[\sqrt{2} \geq \frac{A-B}{\beta+1} \left(F(1,1+\frac{1}{\beta};2+\frac{1}{\beta};-B)\right)+\frac{\beta}{\beta+1}+\frac{1}{\beta+1}.\]
		or equivalently,
		\[
			-\frac{A-B}{\beta+1}\sum_{j=0}^{\infty}\frac{\Gamma (j)}{(j-1)!(1+\beta+j\beta)}B^{j}+\frac{\beta}{\beta+1}+\frac{1}{\beta+1} \geq 0
		\] and
		\[
			\sqrt{2}-\frac{(A-B)}{\beta+1}\sum_{j=0}^{\infty}\frac{\Gamma (j)}{(j-1)!(1+\beta+j\beta)}(-B)^{j}-\frac{\beta}{\beta+1}-\frac{1}{\beta+1} \geq 0.
		\]
		Therefore, the desired subordination $q_{\beta} \prec \sqrt{1+z}$ holds if $\beta \geq \max\{\beta_1, \beta_2\}$, where $\beta_1$ and $\beta_2$ are positive roots of the equations given in (\ref{phil-root}).
		
		\item For $\mathcal{P}(z)=e^{e^{z}-1}$, the subordination $q_{\beta} \prec e^{e^{z}-1}$ holds if
		\[e^{e^{-1}-1} \leq q_{\beta}(-1) \leq q_{\beta}(1) \leq e^{e-1}.\]
		The above inequalities reduces to
		\[e^{e^{-1}-1} \leq -\frac{A-B}{\beta+1} \left(F(1,1+\frac{1}{\beta};2+\frac{1}{\beta};B)\right)+\frac{\beta}{\beta+1}+\frac{1}{\beta+1}\] and
		\[e^{e-1} \geq \frac{A-B}{\beta+1} \left(F(1,1+\frac{1}{\beta};2+\frac{1}{\beta};-B)\right)+\frac{\beta}{\beta+1}+\frac{1}{\beta+1}.\]
		or equivalently,
		\[
			-\frac{A-B}{\beta+1}\sum_{j=0}^{\infty}\frac{\Gamma (j)}{(j-1)!(1+\beta+j\beta)}B^{j}+\frac{\beta}{\beta+1}+\frac{1}{\beta+1}-e^{e^{-1}-1} \geq 0
		\] and
		\[
			e^{e-1}-\frac{(A-B)}{\beta+1}\sum_{j=0}^{\infty}\frac{\Gamma (j)}{(j-1)!(1+\beta+j\beta)}(-B)^{j}-\frac{\beta}{\beta+1}-\frac{1}{\beta+1} \geq 0.
		\]
		Therefore, the required subordination $q_{\beta} \prec e^{e^{z}-1}$ holds if $\beta \geq \max\{\beta_1, \beta_2\}$, where $\beta_1$ and $\beta_2$ are positive roots of the equations given in (\ref{bell-root}).
		
		\item For $\mathcal{P}(z)=1+\tanh z$, the subordination $q_{\beta} \prec 1+\tanh z$ holds if
		\[1+\tanh(-1) \leq q_{\beta}(-1) \leq q_{\beta}(1) \leq 1+\tanh(1).\]
		The above inequalities reduces to
		\[1+\tanh(-1) \leq -\frac{A-B}{\beta+1} \left(F(1,1+\frac{1}{\beta};2+\frac{1}{\beta};B)\right)+\frac{\beta}{\beta+1}+\frac{1}{\beta+1}\] and
		\[1+\tanh(1) \geq \frac{A-B}{\beta+1} \left(F(1,1+\frac{1}{\beta};2+\frac{1}{\beta};-B)\right)+\frac{\beta}{\beta+1}+\frac{1}{\beta+1}.\]
		or equivalently,
		\[
			-\frac{A-B}{\beta+1}\sum_{j=0}^{\infty}\frac{\Gamma (j)}{(j-1)!(1+\beta+j\beta)}B^{j}+\frac{\beta}{\beta+1}+\frac{1}{\beta+1}-1+\tanh(-1) \geq 0
		\] and
		\[
			1+\tanh(1)-\frac{(A-B)}{\beta+1}\sum_{j=0}^{\infty}\frac{\Gamma (j)}{(j-1)!(1+\beta+j\beta)}(-B)^{j}-\frac{\beta}{\beta+1}-\frac{1}{\beta+1} \geq 0.
		\]
		Therefore, the desired subordination $q_{\beta} \prec 1+\tanh z$ holds if $\beta \geq \max\{\beta_1, \beta_2\}$, where $\beta_1$ and $\beta_2$ are positive roots of the equations given in (\ref{tan-root}).
\end{enumerate}
\end{proof}
\begin{remark}
	In particular, when $A=1$ and $B=0$, the above results reduce to \cite[Theorem 2.7]{MR4025222}.
\end{remark}
\begin{corollary}
Let $p(z)=\frac{z f'(z)}{f(z)}$, where $f \in \mathcal{A}$. Assume
\[\frac{z}{f(z)}\left((1+\beta)f'(z)+\beta z\bigg(           f''(z)-\frac{f'(z)^2}{f(z)}\bigg)\right) \prec \frac{1+Az}{1+Bz}.\]
Then,
\begin{enumerate}[(a)]
	\item $f \in S^{*}_{e}$ if $\beta \geq \max\{\beta_1, \beta_2\}$, where $\beta_1$ and $\beta_2$ are positive roots of the equations given in (\ref{e-root}).
	
	\item $f \in S^{*}_{s}$ if $\beta \geq \max\{\beta_1, \beta_2\}$, where $\beta_1$ and $\beta_2$ are positive roots of the equations given in (\ref{sin-root}).
	
	\item $f \in S^{*}_{R}$ if $\beta \geq \max\{\beta_1, \beta_2\}$, where $\beta_1$ and $\beta_2$ are positive roots of the equations given in (\ref{phi0-root}).
	
	\item $f \in S^{*}_{c}$ if $\beta \geq \max\{\beta_1, \beta_2\}$, where $\beta_1$ and $\beta_2$ are positive roots of the equations given in (\ref{phic-root}).
	
	\item $f \in S^{*}_{Ne}$ if $\beta \geq \max\{\beta_1, \beta_2\}$, where $\beta_1$ and $\beta_2$ are positive roots of the equations given in (\ref{phiNe-root}).
	
	\item $f \in S^{*}_{q}$ if $\beta \geq \max\{\beta_1, \beta_2\}$, where $\beta_1$ and $\beta_2$ are positive roots of the equations given in (\ref{phiq-root}).
	
	\item $f \in S^{*}_{L}$ if $\beta \geq \max\{\beta_1, \beta_2\}$, where $\beta_1$ and $\beta_2$ are positive roots of the equations given in (\ref{phil-root}).
	
	\item $f \in S^{*}_{B}$ if $\beta \geq \max\{\beta_1, \beta_2\}$, where $\beta_1$ and $\beta_2$ are positive roots of the equations given in (\ref{bell-root}).
	
	\item $f \in S^{*}_{\tanh}$ if $\beta \geq \max\{\beta_1, \beta_2\}$, where $\beta_1$ and $\beta_2$ are positive roots of the equations given in (\ref{tan-root}).
\end{enumerate}
\end{corollary}
The study of univalent funtions and differential subordination is a very vast domain and it is difficult to obtain all the subordination implications following the same methods. A great deal of research is being done in this direction to explore different ways to solve higher order differential subordination relations. Miller and Mocanu in their monograph mentioned the concept of admissible functions. Recently, authors in \cite{MR4412430} studied first and second order differential subordination using this approach.

\begin{definition}
	Consider the analytic functions having Taylor series expansion $f(z)=a+a_{n}z^{n}+a_{n+1}z^{n+1}+\cdots.$ Then, $\mathcal{H}[a,n]$ denote the class of all such functions for some $a \in \mathbb{C}$ and fixed integer $n$.
\end{definition}

\begin{definition}\cite[Definition 1, p.440]{MR2795467}
	Let $\psi(r,s,t,u;z) : \mathbb{C}^4 \times \mathbb{D} \to \mathbb{D}$ be  analytic and  $h \in \mathscr{U}$.
	Then, the function $p \in \mathcal{A}$, satisfying the third order differential subordination relation
	\[\psi(p(z),zp'(z),z^2p''(z),z^3p'''(z);z) \prec h(z)\]
	is called its \emph{solution}.
\end{definition}
Analogous result for the second order differential subordination is stated as below:
\begin{definition}\cite{MR1760285}
	Let $\psi(r,s,t;z) : \mathbb{C}^3 \times \mathbb{D} \to \mathbb{D}$ be  analytic and  $h$ be a univalent function.
	For $p \in \mathcal{A}$, the subordination relation
	\[\psi(p(z),zp'(z),z^2p''(z);z) \prec h(z)\]
	is known as the second order differential subordination.
\end{definition}
Using the concept of admissibility, authors in \cite{MR3962536} derived the admissibility conditions for the class associated with $exponential$ function as follows:
\begin{definition}
	Let $\Omega \subset \mathbb{C}$ be the domain.
	The class $\Psi_{n}[\Omega;e^{z}]$ is defined as the class of all those functions $\psi:\mathbb{C}^3 \times \mathbb{D} \to \mathbb{C}$
	such that
	\begin{equation*}
		\begin{aligned}
			&\psi(r,s,t;z) \not\in \Omega \enspace \text{whenever} \enspace (r,s,t;z) \in \operatorname{Dom} \psi, \\
			&r=e^{e^{i \theta}}, s=m e^{i \theta}r, \operatorname{Re}\left(1+\frac{t}{s}\right) \geq m(1+\cos \theta),
		\end{aligned}
	\end{equation*}
	for $z \in \mathbb{D}, \theta \in (0, 2\pi)$ and $m \geq 1$.
\end{definition}

Extending the existing work done and using the results given in \cite{MR2795467}, we formulated an additional condition involving third order differential subordination parameter.
Henceforth, the admissibility conditions for the $exponential$ function are restated as follows:
\begin{definition}
        Let $\Omega \subset \mathbb{C}$ and $n \geq 1$.
		For $q(z)=e^{z}$, the admissibility conditions are given as follows:
		\begin{equation*}
		\begin{aligned}
			&\psi(r,s,t,u;z) \not\in \Omega \enspace \text{whenever} \enspace (r,s,t,u;z) \in \operatorname{Dom} \psi, \\
			&r=e^{e^{i \theta}}, s=m e^{i \theta}r, \operatorname{Re}\left(1+\frac{t}{s}\right) \geq m(1+\cos \theta), \\
			&\text{ and } \operatorname{Re}\left(\frac{u}{s}\right) \geq m\cos 2 \theta
		\end{aligned}
	\end{equation*}
		for $z \in \mathbb{D}, \theta \in (0, 2\pi)$, $\cos 2\theta \geq 0$ and $m \geq 1$.
\end{definition}

\begin{theorem}\cite{MR3962536}\label{admissibility-secondorder}
Let $p \in \mathcal{H}[1,n]$ and $\Omega$ be a set in $\mathbb{C}$.
If $\psi \in \Psi[\Omega;e^z]$ and
\[\psi(p(z),zp'(z),z^2p''(z);z) \subset \Omega.\]
Then, $p(z)$ is subordinate to $e^{z}$.
\end{theorem}

Motivated by the above mentioned definitions and theorems, the following subordination implication is established. In the following two results, we have taken $q(z)=e^{z}$.
\begin{theorem}
Let $p(z)$ denote an analytic function with $p(0)=1$ and $\alpha \geq e(e-1)+1$. Then,
\[|1+\alpha zp'(z)+ z^2p''(z)|<e \implies p(z) \prec e^{z}.\]
\end{theorem}
\begin{proof}
Set $\Omega=\{w \in \mathbb{C}: |\log w|<1\}$ and define the function $\psi: \mathbb{C}^{3} \times \mathbb{D} \rightarrow \mathbb{C}$ by $\psi(r,s,t;z)=1+\alpha s+ t$.
Thus, using the fact $|\log(1+z)| \geq 1$ if and only if $|z| \geq e-1$, we have
\begin{align*}
	|\psi(r,s,t;z)-1|&=|\alpha s+ t|\\
	&= |s|\left|\alpha +\frac{t}{s}\right|\\
	&\geq e^{-1}\operatorname{Re}\left(\alpha+\frac{t}{s}\right)\\
	&\geq e^{-1}((\alpha -1) +m(1+\cos \theta))\\
	&\geq e^{-1}(\alpha-1)\\
	&\geq e-1.
\end{align*}
Hence, $|\psi-1| \geq e-1$ implies $|\log \psi| \geq 1$ further implying $\psi \notin \Psi_{n}[\Omega;e^{z}]$.
\end{proof}

Insightful work has been done by several authors on second-order differential subordination. However, subordination implications involving third and higher order deriatives are still limited.
For few recent works done in this direction, refer~\cite{MR3804281,MR4449282,MR1232447}.
In the following theorems, sufficient conditions on certain non-negative real numbers in certain differential subordination implications are obtained so that any analytic function $f$ is $exponential$ starlike in $\mathbb{D}$.
We need the following definition and theorems in order to prove our claimed results.

\begin{definition}\cite{MR1760285}
Let $\mathbf{E}(q)$ be the collection of all those points $\zeta \in \partial \mathbb{D}$ such that $q(z) \rightarrow \infty$ as $z \rightarrow \zeta$.
Then, $\mathcal{Q}$ denote the class of all functions $q$, which are analytic and univalent on $\overline{\mathbb{D}} \setminus \mathbf{E}(q)$.
\end{definition}

\begin{theorem}\cite{MR2795467}
	Assume $p \in \mathcal{H}[a,n]$ with $n \geq 2$ and $q \in \mathcal{Q}(a)$ satisfying the inequalities
	\[\operatorname{Re}\frac{wq''(z)}{q'(z)} \geq 0 \quad \text{and} \quad \left|\frac{zp'(z)}{q'(w)}\right| \leq k,\]
	for $k \geq n$, $z \in \mathbb{D}$ and $w \in \partial\mathbb{D} \setminus \mathbf{E}(q)$. If $\Omega \subset \mathbb{C}$, $\psi \in \Psi_{n}[\Omega;q]$ and
	\[\psi(p(z),zp'(z),z^2p''(z),z^3p'''(z);z) \subset \Omega,\]
	then $p$ is subordinate to $q$.
\end{theorem}
For $q(z)=e^{z}$, the following result is a particular case of above theorem.
\begin{theorem}\label{admissibility-thirdorder}
	Assume $p \in \mathcal{H}[1,n]$ with $n \in \mathbb{N}$ satisfy the inequality
	\[|zp'(z)| \leq n,\]
	for $z \in \mathbb{D}$ and $w \in \partial\mathbb{D} \setminus \mathbf{E}(q)$. If $\Omega \subset \mathbb{C}$, $\psi \in \Psi_{n}[\Omega;q]$ and
	\[\psi(p(z),zp'(z),z^2p''(z),z^3p'''(z);z) \subset \Omega,\]
	then $p(z) \prec e^{z}$.
\end{theorem}
\begin{theorem}
	Let $0<A \leq 1$ and let $\alpha, \beta, \gamma$ be non-negative real numbers.
	Then, the following relations are sufficient for $p(z) \prec e^{z}$.
	\begin{enumerate}[(i)]
		\item $1+\alpha z^3 p'''(z)+\beta z^2p''(z)+\gamma z p'(z) \prec 1+Az$, where
		\[\gamma \geq
		\begin{cases}
			Ae+\alpha+\beta^2/8\alpha, & \text{when } \beta \leq 4 \alpha \\
			Ae-\alpha+\beta, & \text{when } \beta \geq 4 \alpha.
		\end{cases}\]
		\item $1+\alpha \left(z^3 p'''(z)/z p'(z)\right)+\beta \left(z p'(z)/p(z)+1\right) \prec 1+Az$, where
		\[\begin{cases}
			\beta-\alpha-\beta^2/8\alpha \geq A, & \text{when } \beta \leq 4 \alpha \\
			\alpha \geq A, & \text{when } \beta \geq 4 \alpha.
		\end{cases}\]
	\end{enumerate}
\end{theorem}
\begin{proof}
	Let $\Omega=\{w \in \mathbb{C}: |w-1|<A\}$ be the domain.
	\begin{enumerate}[(i)]
		\item We define the function $\psi: \mathbb{C}^{4} \times \mathbb{D} \rightarrow \mathbb{C}$ as $\psi(r,s,t,u;z)=1+\alpha u+\beta t+ \gamma s$. The associated admissibility conditions are satisfied whenever $\psi \notin \Omega$ and thus, simple computations gives
	\begin{align}\label{eq1}
		|\psi(r,s,t,u;z)-1|&=|\alpha u+\beta t+ \gamma s|\nonumber\\
		&= |s|\left|\alpha \frac{u}{s}+\beta \left(\frac{t}{s}+1\right)+\gamma-\beta \right|\nonumber\\
		&\geq e^{-1}\left|\alpha \frac{u}{s}+\beta \left(\frac{t}{s}+1\right)+\gamma-\beta \right|\nonumber\\
		&\geq e^{-1}\left\{\alpha \operatorname{Re}\frac{u}{s}+\beta \operatorname{Re}\left(\frac{t}{s}+1\right)+\gamma -\beta\right\}\nonumber\\
		&\geq e^{-1}(\alpha \cos 2 \theta+\beta m(1+\cos \theta)+\gamma-\beta)\nonumber\\
		&\geq e^{-1}(\alpha \cos 2 \theta+\beta (1+\cos \theta)+\gamma-\beta).
	\end{align}
	For $x=\cos \theta$, let $g(x):=\alpha (2x^2-1)+\beta (1+x)+\gamma-\beta$.
	In order to find the minimum value of $g(x)$, there are two possibilities:
	\begin{description}
		\item[Case 1] $\beta \leq 4 \alpha$. It is easy to see that $g'(x) = 0$ if $x=-\beta /4 \alpha$. By the second derivative test, it follows that
		\begin{equation}\label{eq2}
\min_{|x| \leq 1}g(x)=g(-\beta/4 \alpha)=\gamma-\alpha-\beta^2/8\alpha.
\end{equation}
		By \eqref{eq1} and \eqref{eq2}, we have
		\[|\psi(r,s,t,u;z)-1| \geq e^{-1}\left(\gamma-\alpha-\beta^2/8\alpha\right) \geq A\]
		and hence, the result follows from Theorem (\ref{admissibility-thirdorder}).
		
		\item[Case 2]  $\beta > 4 \alpha$. Since $-\beta/(4\alpha)\leq-1$,  $g'(x) \neq 0$.
		It yields that the function $g$ is increasing and
		\begin{equation}\label{eq3}
    \min_{|x| \leq 1}g(x)=g(-1)=\alpha-\beta+\gamma.
    \end{equation}
	    By \eqref{eq1} and \eqref{eq3}, we have
		\[|\psi(r,s,t,u;z)-1| \geq e^{-1}\left(\alpha-\beta+\gamma\right) \geq A\]
		and hence, the result follows from Theorem (\ref{admissibility-thirdorder}).
		\end{description}
		
		\item Let $\psi(r,s,t,u;z)=1+\alpha \left(u/s\right)+\beta \left(s/r+1\right)$. Observe that
		\begin{align}\label{eq4}
			|\psi(r,s,t,u;z)-1|&=\left|\alpha \left(\frac{u}{s}\right)+\beta \left(\frac{s}{r}+1\right)\right|\nonumber\\
			&\geq \beta+\beta \operatorname{Re}\left(\frac{s}{r}\right)+\alpha\operatorname{Re}\left(\frac{u}{s}\right)\nonumber\\
			&\geq \beta (1+\cos \theta)+\alpha \cos 2\theta.
		\end{align}
	For $x=\cos \theta$, let $m(x):=\beta (1+x)+\alpha (2x^2-1)$. To find the minimum value of the function $g$, the following two cases arises:
	\begin{description}
		\item[Case 1] $\beta \leq 4 \alpha$. An easy calculation shows that  $m'(x) = 0$ if and only if $x=-\beta/(4\alpha)$. The second derivative test for extrema gives
		\begin{equation}\label{eq5}
        \min_{|x| \leq 1}m(x)=m(-\beta/4 \alpha)=\beta-\alpha-\beta^2/8\alpha.
        \end{equation}
		Equations \eqref{eq4} and \eqref{eq5} yeild
		\[|\psi(r,s,t,u;z)-1| \geq \beta-\alpha-\beta^2/8\alpha \geq A\]
		and hence, the result follows from Theorem (\ref{admissibility-thirdorder}).
		
		\item[Case 2]  $\beta > 4 \alpha$. Then $-\beta/(4\alpha)<-1$. So, the function $m(x)$ is increasing and the minimum value is
		\[\min_{|x| \leq 1}m(x)=m(-1)=\alpha.\]
		Thus,
		\[|\psi(r,s,t,u;z)-1| \geq \alpha \geq A\]
		and hence, the result follows from Theorem (\ref{admissibility-thirdorder}).
	\end{description}
\end{enumerate}
\end{proof}


\begin{theorem}
	Let $\alpha, \beta, \gamma$ be non-negative real numbers.
	Then, the following relations are sufficient for $p(z) \prec e^{z}$.
	\begin{enumerate}[(i)]
		\item $1+\alpha z^3 p'''(z)+\beta z^2p''(z)+\gamma z p'(z) \prec e^z$, where
		\[\gamma \geq
		\begin{cases}
			e(e-1)+\alpha+\beta^2/8\alpha, & \text{when } \beta \leq 4 \alpha \\
			e(e-1)-\alpha+\beta, & \text{when } \beta \geq 4 \alpha.
		\end{cases}\]
		\item $1+\alpha \left(z^3 p'''(z)/z p'(z)\right)+\beta \left(z p'(z)/p(z)+1\right) \prec e^z$, where
		\[
		\begin{cases}
			\beta \geq \alpha+\beta^2/8\alpha+e-1, & \text{when } \beta \leq 4 \alpha \\
			\alpha \geq e-1, & \text{when } \beta \geq 4 \alpha.
		\end{cases}\]
	\end{enumerate}
\end{theorem}
\begin{proof}
	Let $\Omega=\{w \in \mathbb{C}: |\log w|<1\}$ be the domain.
	\begin{enumerate}[(i)]
		\item Let $\psi(r,s,t,u;z)=1+\alpha u+\beta t+ \gamma s$.
		Making use of the fact $|\log(1+z)| \geq 1$ if and only if $|z| \geq e-1$, we have
		\begin{align*}
			|\psi(r,s,t,u;z)-1|&=|\alpha u+\beta t+ \gamma s|\\
			&= |s|\left|\alpha \frac{u}{s}+\beta \left(\frac{t}{s}+1\right)+\gamma-\beta \right|\\
			&\geq e^{-1}\left|\alpha \operatorname{Re}\frac{u}{s}+\beta \operatorname{Re}\left(\frac{t}{s}+1\right)+\gamma -\beta\right|\\
			&\geq e^{-1}\left\{\alpha \operatorname{Re}\frac{u}{s}+\beta \operatorname{Re}\left(\frac{t}{s}+1\right)+\gamma -\beta\right\}\\
			&\geq e^{-1}(\alpha \cos 2 \theta+\beta (1+\cos \theta)+\gamma-\beta).
		\end{align*}
	Let $x=\cos \theta$. As in the previous theorem, we consider the following two cases to find the minimum value of the function $g(x):=\alpha \cos(2x^2-1)+\beta (1+\cos x)+\gamma-\beta$.
	\begin{description}
		\item[Case 1] $\beta \leq 4 \alpha$. It is easy to see that $g'(x) = 0$ if $x=-\beta /4 \alpha$. Thus,
the minimum value of the function $g$ occurs at $x=-\beta/4 \alpha$ and hence,
		\[|\psi(r,s,t,u;z)-1| \geq e^{-1}\left(\gamma-\alpha-\beta^2/8\alpha\right) \geq e-1.\]
		Therefore, the result follows from Theorem (\ref{admissibility-thirdorder}).
		
		\item[Case 2]  $\beta > 4 \alpha$. Since $-\beta/(4\alpha)\leq-1$,  $g'(x) \neq 0$.
		It yields that the function $g$ is increasing and thus,
		\[|\psi(r,s,t,u;z)-1| \geq e^{-1}(\alpha-\beta+\gamma) \geq e-1\]
		and hence, the result follows from Theorem (\ref{admissibility-thirdorder}).
	\end{description}
		
		\item Let $\psi(r,s,t,u;z)=1+\alpha \left(u/s\right)+\beta \left(s/r+1\right)$.
		Using the fact $|\log(1+z)| \geq 1$ if and only if $|z| \geq e-1$, we have
		\begin{align*}
			|\psi(r,s,t,u;z)-1|&=\left|\alpha \left(\frac{u}{s}\right)+\beta \left(\frac{s}{r}\right)\right|\\
			&\geq \beta+ \beta \operatorname{Re}\left(\frac{s}{r}\right)+\alpha\operatorname{Re}\left(\frac{u}{s}\right)\\
			&\geq \beta (1+\cos \theta)+\alpha \cos 2\theta.
		\end{align*}
		For $x=\cos \theta$, define the function $m(x):=\beta (1+x)+\alpha (2x^2-1)$.
		\begin{description}
			\item[Case 1] $\beta \leq 4 \alpha$. An easy calculation shows that  $m'(x) = 0$ if and only if $x=-\beta/(4\alpha)$. The second derivative test for extrema gives   the function $m(x)$ attains its minimum value at $x=-\beta/4 \alpha$ and thus, we have
			\[|\psi(r,s,t,u;z)-1| \geq \beta-\alpha-\beta^2/8\alpha \geq e-1.\]
			Hence, the result follows from Theorem (\ref{admissibility-thirdorder}).
			\item[Case 2]   $\beta > 4 \alpha$. Then $-\beta/(4\alpha)<-1$. So, the function $m(x)$ is increasing and  thus,
			\[|\psi(r,s,t,u;z)-1| \geq \alpha \geq e-1\]
			and hence, the result follows from Theorem (\ref{admissibility-thirdorder}).
		\end{description}\qedhere
\end{enumerate}
\end{proof}

\begin{theorem}
	Let $\alpha, \beta, \gamma$ be non-negative real numbers and $\beta_{0} \approx 0.475319$.
	Then, the following relations are sufficient for $p(z) \prec e^{z}$.
	\begin{enumerate}[(i)]
		\item $\alpha z^3 p'''(z)+\beta z^2p''(z)+\gamma z p'(z) \prec \phi_{SG}(z)$, where
		\[\gamma \geq
		\begin{cases}
			\beta_{0}e+\alpha+\beta^2/8\alpha, & \text{when } \beta \leq 4 \alpha \\
			\beta_{0}e-\alpha+\beta, & \text{when } \beta \geq 4 \alpha.
		\end{cases}\]
		\item $\alpha \left(z^3 p'''(z)/z p'(z)\right)+\beta \left(z p'(z)/p(z)+1\right) \prec \phi_{SG}(z)$, where
		\[
		\begin{cases}
		\beta \geq \beta_{0}+\alpha+\beta^2/8\alpha, & \text{when } \beta \leq 4 \alpha \\
		\alpha \geq	\beta_{0}, & \text{when } \beta \geq 4 \alpha.
		\end{cases}\]
	\end{enumerate}
\end{theorem}
\begin{proof}
	Let $\Omega=\{w \in \mathbb{C}: \left|\log\left(w/(2-w)\right)\right|<1\}$ be the domain.
	\begin{enumerate}[(i)]
		\item Define the function $\psi(r,s,t,u;z)=\alpha u+\beta t+ \gamma s$.
		Using \cite[Lemma 2.7]{MR4412430}, which says $\left|\log\left(z/(2-z)\right)\right| \geq 1$ if and only if $|z| \geq \beta_{0}$, we have
		\begin{align*}
			|\psi(r,s,t,u;z)|&=|\alpha u+\beta t+ \gamma s|\\
			&= |s|\left|\alpha \frac{u}{s}+\beta \left(\frac{t}{s}+1\right)+\gamma-\beta \right|\\
			&\geq e^{-1}\left|\alpha \frac{u}{s}+\beta \left(\frac{t}{s}+1\right)+\gamma-\beta \right|\\
			&\geq e^{-1}\left\{\alpha \operatorname{Re}\frac{u}{s}+\beta \operatorname{Re}\left(\frac{t}{s}+1\right)+\gamma -\beta\right\}\\
		    &\geq e^{-1}(\alpha \cos 2 \theta+\beta (1+\cos \theta)+\gamma-\beta).
		\end{align*}
	For $x=\cos \theta$, let us consider $g(x):=\alpha (2x^2-1)+\beta (1+x)+\gamma-\beta$.
	\begin{description}
		\item[Case 1] $\beta \leq 4 \alpha$. Observe that $g'(x) = 0$ if $x=-\beta /4 \alpha$. Since
the minimum value of the function $g$ occurs at $x=-\beta/4 \alpha$, we have
		\[|\psi(r,s,t,u;z)| \geq e^{-1}\left(\gamma-\alpha-\beta^2/8\alpha\right) \geq \beta_{0}.\]
		Hence, the required result follows from Theorem (\ref{admissibility-thirdorder}).
		
		\item[Case 2] $\beta > 4 \alpha$. Since $-\beta/(4\alpha)\leq-1$,  $g'(x) \neq 0$, it follows that
		 the function $g$ is increasing and thus,
\[|\psi(r,s,t,u;z)| \geq e^{-1}(\alpha-\beta+\gamma) \geq \beta_{0}\]
		and hence, the result follows from Theorem (\ref{admissibility-thirdorder}).
	\end{description}
	
		\item Consider the function $\psi(r,s,t,u;z)=\alpha \left(u/s\right)+\beta \left(s/r+1\right)$.
	    Using \cite[Lemma 2.7]{MR4412430}, which says $\left|\log\left(z/(2-z)\right)\right| \geq 1$ if and only if $|z| \geq \beta_{0}$, we have
		\begin{align*}
			|\psi(r,s,t,u;z)|
			&=\left|\alpha \left(\frac{u}{s}\right)+\beta \left(\frac{s}{r}+1\right)\right|\\
			&\geq \beta+\beta \operatorname{Re}\left(\frac{s}{r}\right)+\alpha\operatorname{Re}\left(\frac{u}{s}\right)\\
			&\geq \beta (1+\cos \theta)+\alpha \cos 2\theta.
		\end{align*}
        For $x=\cos \theta$, the function $m(x):=\beta (1+x)+\alpha (2x^2-1)$ will attain its minimum value in the following two cases:
	
		\begin{description}
		\item[Case 1]$\beta \leq 4 \alpha$. Note that the function $m$ attains its minimum value at $x=-\beta/4 \alpha$ and thus, we have
		\[|\psi(r,s,t,u;z)| \geq \beta-\alpha-\beta^2/8\alpha \geq \beta_0.\]
		Henceforth, the desired result follows from Theorem (\ref{admissibility-thirdorder}).
		
		\item[Case 2] $\beta > 4 \alpha$. Then $-\beta/(4\alpha)<-1$. Thus, the function $m(x)$ is increasing and
		\[|\psi(r,s,t,u;z)| \geq \alpha \geq \beta_{0}\]
		and hence, the result follows from Theorem (\ref{admissibility-thirdorder}).
	\end{description}
\end{enumerate}
\end{proof}

\begin{corollary}
	As an application, we obtain sufficient conditions for the function $f \in \mathcal{A}$ to be exponential starlike in $\mathbb{D}$ by substituting $p(z) = zf'(z)/f(z)$. 	
\end{corollary}

\end{document}